\documentclass{amsproc}

\usepackage{amscd,amssymb,eucal,mathrsfs,amsmath}
\usepackage{hyperref} 
\usepackage{graphics,colortbl}

\input xy
\xyoption{all}
\usepackage{epsfig}

\newtheorem{thm}{Theorem}[section]
\newtheorem{cor}[thm]{Corollary}
\newtheorem{lem}[thm]{Lemma}
\newtheorem{prop}[thm]{Proposition}

\newtheorem{question}[thm]{Question}                                     
\theoremstyle{definition}
\newtheorem{defin}[thm]{Definition}

\theoremstyle{remark}
\newtheorem{remark}[thm]{Remark}

\newtheorem{examples}[thm]{Examples}
\numberwithin{equation}{section}
\newtheorem{problem}[thm]{Problem}

\numberwithin{equation}{section}

\newcommand{\delete}[1]{} 
\newcommand{\nt}{\noindent}

\def\ep{{\varepsilon}}
\def\al{\alpha}

\newcommand{\g}{\gamma}
\newcommand{\del}{\delta}

\newcommand{\sk}{\vskip 0.2cm}
\newcommand{\ssk}{\vskip 0.1cm}

\newcommand{\ben}{\begin{enumerate}}
	
	\newcommand{\een}{\end{enumerate}}
\newcommand{\bit}{\begin{itemize}}
	
	\newcommand{\eit}{\end{itemize}}

\def\R {{\mathbb R}}
\def\N {{\mathbb N}}
\def\Z {{\mathbb Z}}
\def\Q {{\mathbb Q}}

\def\T {{\mathbb T}}

\def\Aut{{\mathrm Aut}\,}

\def\Homeo{{\mathrm{Homeo}}\,}

\def\Sp{{\mathrm{Split}}}
\def\Aut{{\rm{Aut}}}

\newcommand{\abs}[1]{\lvert#1\rvert}

\begin{document}

\title[]{Circular orders, ultra-homogeneous order structures and their automorphism groups}

\author[]{Eli Glasner}
\address{Department of Mathematics,
	Tel-Aviv University, Ramat Aviv, Israel}
\email{glasner@math.tau.ac.il}
\urladdr{http://www.math.tau.ac.il/$^\sim$glasner}

\author[]{Michael Megrelishvili}
\address{Department of Mathematics,
	Bar-Ilan University, 52900 Ramat-Gan, Israel}
\email{megereli@math.biu.ac.il}
\urladdr{http://www.math.biu.ac.il/$^\sim$megereli}
\thanks{This research was supported by a grant of the Israel Science Foundation (ISF 1194/19) 
and also by the Gelbart Research Institute at the Department of Mathematics, Bar-Ilan  University}  




\begin{abstract}  
	We study topological groups $G$ for which the universal minimal $G$-system $M(G)$, or the universal irreducible affine $G$-system $I\!A(G)$ are tame. 
	We call such groups ``intrinsically tame" and ``convexly intrinsically tame", respectively. These notions, which were introduced in \cite{GM-tLN}, are generalized versions of extreme amenability and amenability, respectively. 
	When $M(G)$, as a $G$-system, admits a circular order 
	we say that $G$ is intrinsically circularly ordered.
	This implies that $G$ is intrinsically tame.
	
	We show that given a circularly ordered set $X_\circ$, any subgroup $G \leq \Aut(X_\circ)$ whose action on $X_\circ$ is ultrahomogeneous, when equipped with the topology $\tau_p$ of pointwise convergence, is intrinsically circularly ordered. This result is a ``circular" analog of Pestov's result about the extreme amenability of ultrahomogeneous actions on linearly ordered sets by linear order preserving transformations. We also describe, for such groups $G$, the dynamics of the system $M(G)$, and show that it is extremely proximal (whence $M(G)$ coincides with the universal strongly proximal $G$-system), and deduce that the group $G$ must contain a non-abelian free group.
	
	In the case where $X$ is countable, the corresponding Polish group of circular automorphisms $G=\Aut(X_o)$ admits a concrete description.
	Using the Kechris-Pestov-Todorcevic construction we show that $M(G)=\Sp(\T;\Q_{\circ})$, a circularly ordered compact metric space (in fact, a Cantor set) obtained by splitting the rational points on the circle $\T$.  
	We show also that $G=\Aut(\Q_{\circ})$ is Roelcke precompact,  satisfies Kazhdan's property $T$ (using results of Evans-Tsankov) and has the automatic continuity property (using results of Rosendal-Solecki).  
\end{abstract}

\maketitle

\setcounter{tocdepth}{1}
\tableofcontents

\section{Introduction}

The universal minimal $G$-system $M(G)$ of a topological group $G$ 
can serve as a beautiful link between the theories of topological groups and topological dynamics. 
In general, $M(G)$ may be very large. For instance, it is 
always nonmetrizable for locally compact noncompact groups. 
The question whether $M(G)$ is small in some sense or another is pivotal in this theory. 
Recall that $G$ is said to be \textit{extremely amenable} (or to have the {\em fixed point on compacta property}) when the system $M(G)$ is trivial. 
An important and now well studied (see \cite{BMT17, van-theARXIV, Zucker16}) 
question is : when is $M(G)$ metrizable~? 

Another interesting new direction is to determine when is $M(G)$ \textit{dynamically small} 
or non-chaotic.
A closely related question is: when is the universal irreducible affine $G$-system $I\!A(G)$ (dynamically) small ? It is well known that $G$ is amenable iff $I\!A(G)$ is trivial. 

In \cite{GM-tLN} we raised the question: when are $M(G)$ and $I\!A(G)$ tame dynamical systems~? 
In the present work we give a new sufficient condition which guarantees that $M(G)$ and $I\!A(G)$ are circularly ordered $G$-systems (hence, by our previous work \cite{GM-c}, tame).   

Recall that an order preserving 
effective  
action $G \curvearrowright X$ on a linearly ordered infinite set $X_{<}$ is said to be \emph{ultrahomogeneous} if every order isomorphism between two finite subsets can be extended to an order automorphism $g \in G$ of $X_{<}$.   
We say that $X_{<}$ is  \emph{ultrahomogeneous}
when the tautological action of $\Aut(X_{<})$ on $X$ is ultrahomogeneous.

As was shown by Pestov \cite{Pes98}, 
for any subgroup $G \leq \Aut(X_{<})$ 
whose action on $X_{<}$ is ultrahomogeneous, the topological group $(G,\tau_p)$, 
in its pointwise convergence topology, is extremely amenable. 
That is, every continuous action of $G$ on a compact Hausdorff space has a fixed point. For example the Polish group $\Aut(\Q_{<})$ of automorphisms of 
the linearly ordered set $\Q$ is extremely amenable in its pointwise convergence topology. 
Ultrahomogeneous structures, Ramsey theory and Fra\"{i}ss\'{e} limits play a major role in many modern works, \cite{Pes98}, \cite{GW-02}, \cite{GW-03}, \cite{KPT,Pe-b,van-the,MNT}.  

Our aim here is to examine the role of circular orders (\emph{c-order}, for short). 
We say that an effective action of a group $G$ of c-order automorphisms on a c-ordered infinite set $X_{\circ}$ is \emph{ultrahomogeneous} if every c-order isomorphism between two finite subsets 
can be extended to a c-order automorphism of $X_{\circ}$ in $G$. 
We say that $X_{\circ}$ is \emph{ultrahomogeneous} when the tautological action of $\Aut(X_{\circ})$
on $X$ is ultrahomogeneous. 

An important particular case is $X=\Q_{\circ}=\Q / \Z$, the rational points of the 
circle $\T=\R / \Z$. The corresponding automorphism group $G:=\Aut(\Q_{\circ})$, 
equipped with the topology of pointwise convergence $\tau_p$, 
is a Polish nonarchimedean topological group (hence a closed subgroup of the symmetric group $S_{\infty}$).

Let $G$ be a topological group and $\textbf{P}$ a property of $G$-dynamical systems, we say
$G$ has the {\em intrinsic property $\textbf{P}$} when the universal minimal $G$-system $M(G)$
has this property.
Of course when $\textbf{P}$ is a property which is preserved by factors, this is the same as saying
that every minimal $G$-system has the property $\textbf{P}$.
With $\textbf{P}$ the property of being the trivial one point system, having the intrinsic $\textbf{P}$ property
reduces to extreme amenability. Our general idea is to relax this extreme case
by choosing $\textbf{P}$ to be a ``small" class of $G$-systems (see \cite{GM-tLN}).

In Theorem \ref{main-ultra} we prove that 
when a group $G$ acts ultrahomogeneously 
on an infinite circularly ordered set $X$, 
then $G$, with the pointwise convergence topology it inherits from this action, 
is intrinsically c-ordered.  
That is, the universal minimal $G$-system $M(G)$ is a circularly ordered $G$-system. 
In particular this applies to $\Aut(X_{\circ})$ when $X_{\circ}$ is ultrahomogeneous. 
We do not know whether this implies that every minimal $G$-system is circularly ordered.

Furthermore, in the case where $X$ is countable, it follows that
$M(G)$ is a metrizable circularly ordered $G$-system. 
In particular this is true for $G=\Aut(\Q_{\circ})$. 
This theorem also applies to  
Thompson's (finitely generated, nonamenable) circular group $T$, which 
acts ultrahomogeneously on $D_{\circ}$, the set of dyadic rationals in the circle.  

Again the idea is to regard the requirement that $M(G)$ 
be a c-ordered set as a relaxation of the requirement that $M(G)$ be
a trivial system.
In this sense Theorem \ref{main-ultra} is indeed an analog of Pestov's theorem because 
every linearly ordered minimal compact $G$-system is necessarily trivial.

In contrast to linear orders, the class of minimal compact circularly ordered $G$-spaces 
is quite large. For instance, the study of minimal subsystems of the circle $\T$ 
with an action defined by 
a c-order preserving homeomorphism goes back to classical works of Poincar\'{e}, 
Denjoy and Markley. 
In symbolic dynamics we have extensive studies of 
Sturmian like systems which are c-ordered, \cite{GM-c}. 
Finally, by another theorem of Pestov, 
the universal minimal system $M(H_+(\T))$ of the Polish group $H_+(\T)$, 
of c-order preserving homeomorphisms of the circle $\T$, 
is $\T$ itself with the tautological action,~\cite{Pe-b}.  

In Theorem \ref{thm;ep} we describe, for groups $G$ as in Theorem \ref{main-ultra},
the dynamics of the system $M(G)$, and show that
it is extremely proximal (whence $M(G)$ is the universal strongly proximal $G$-system).
We also deduce that the group $G$ must contain a non-abelian free group.

\vspace{.3cm}

One of the reasons we think of a c-ordered dynamical system as being a relaxation of being trivial is the fact we proved in \cite{GM-c}, that every 
c-ordered compact, not necessarily metrizable, $G$-space $X$
is tame. In fact we prove there a stronger result, namely that such a system 
can always be represented on a \emph{Rosenthal Banach space}. 
We refer to \cite{GM-rose,GM-c,GM-tLN} for more information about tame dynamical systems. 
See also Remark \ref{r:collapsing}. 

\vspace{.3cm}

In Section \ref{Sec:gener} we prepare the ground for our main 
results which are proved in Section \ref{Sec:main}.  
In Section \ref{Fraisse} we 
present an alternative proof of Theorem \ref{main-ultra}, in the countable case.
It employs the notion of Fra\"{i}ss\'{e} classes and the Kechris-Pestov-Todorcevic theory \cite{KPT} and has the advantage that it automatically yields an explicit description of $M(G)$.
Namely, in terms of \cite{GM-c}, $M(\Aut(\Q_{\circ}))$ is 
the system $\Sp(\T; \Q_\circ)$ obtained  from  the circle by 
splitting in two the rational points. 

In Section \ref{AC} we examine some other properties of 
the Polish group 
$\Aut(\Q_{\circ})$. 
This topological group
has the automatic continuity property (Lemma 
\ref{lem-automat}). This easily follows from the automatic continuity property of 
$\Aut(\Q_{<})$ established by Rosendal and Solecki \cite{RS}. As an application we get that the discrete group $\Aut(\Q_{\circ})$ is 
{\em metrically intrinsically c-ordered} 
(see Definition  \ref{d:int-tame} below).  
That is, every action of $G:=\Aut(\Q_{\circ})$ by homeomorphisms on a metric compact space, admits a closed $G$-invariant subsystem which 
is a circularly ordered $G$-subsystem, Corollary \ref {metr-conv-int}.

We also note that 
the group $\Aut(\Q_{\circ})$
is Roelcke precompact (Proposition \ref{t:Rolelckeprecompact}); 
and finally, a result of  Evans and Tsankov \cite{ET} implies that it has 
Kazhdan's property (T)
(see Corollary \ref{c:kazhdan}).

\sk

\nt \textbf{Acknowledment.} 
We would like to express our thanks to Vladimir Pestov for his 
influence and inspiration;  as well as for sending us the unpublished preprint \cite{Pest-pr18}.   
We are grateful to J. K. Truss for providing us the unpublished preprint \cite{Tr}, and to 
G. Golan for her comments about Thompson groups.

\sk

\section{Some generalizations of (extreme) amenability} \label{Sec:gener}

An action $G \times X \to X$ of a group $G$ on a set $X$ is \textit{effective} if $gx=x \ \forall x \in X$ is possible only for $g=e$, where $e$ is the identity of $G$.  
Below all compact spaces are Hausdorff. 
A compact space $X$ with a given continuous action $G \curvearrowright X$ of a topological group $G$ on $X$ is said to be a \textit{dynamical $G$-system}. 
 
Recall the classical definition from \cite{Ros0}: a sequence $f_n$ of real valued functions on a set $X$ 
is said to be \emph{independent} if
there exist real numbers $a < b$ such that
$$
\bigcap_{n \in P} f_n^{-1}(-\infty,a) \cap  \bigcap_{n \in M} f_n^{-1}(b,\infty) \neq \emptyset
$$
for all finite disjoint subsets $P, M$ of $\N$.

By A. K\"{o}hler's \cite{Koh} definition, a dynamical $G$-system $X$ is \textit{tame} if  
for every continuous real valued function 
$f: X \to \R$ the family of functions $fG:=\{fg\}_{g \in G}$ 
does not contain an {\it independent sequence}.

The following result is a dynamical analog
of a well known Bourgain-Fremlin-Talagrand dichotomy \cite{BFT}.

\begin{thm} \label{D-BFT}
	\cite{GM1} 
	Let $X$ be a compact metric dynamical $S$-system and let $E=E(X)$ be its
	enveloping semigroup. We have the following alternative. Either
	\begin{enumerate}
		\item
		$E$ is a separable Rosenthal compact 
		
		\nt (hence $E$ is Fr\'echet and ${card} \; {E} \leq
		2^{\aleph_0}$); or
		\item
		the compact space $E$ contains a homeomorphic
		copy of $\beta\N$ 
		
		\nt (hence ${card} \; {E} = 2^{2^{\aleph_0}}$).
	\end{enumerate}
	The first possibility holds iff $X$ is a tame $S$-system.
\end{thm}

A dynamical $G$-system $X$ is said to be circularly (linearly) ordered if $X$ is a circularly (linearly) ordered space and each element of $G$ preserves the circular (linear) order,  \cite{GM-c}. 
In Section \ref{s:c} we give some background about circular order. 

A topological group $G$ is said to be \textit{nonarchimedean} if it has a
base of open neighbourhoods of the identity consisting of (clopen) subgroups. 
Equivalently, the groups which can be embedded into the symmetric groups $S_X$. 

As usual, $M(G)$ denotes the \textit{universal minimal $G$-system} of a topological group $G$. 
By $I\!A(G)$ we denote the \textit{universal irreducible affine $G$-system}, \cite{Gl-book1}. 
Recall that a topological group $G$ is \textit{amenable} (\textit{extremely amenable}) iff $I\!A(G)$ 
(respectively, $M(G)$) is trivial. The following definition proposes some generalizations.

\begin{defin} \label{d:int-tame} \cite{GM-tLN}
	A topological group $G$ is said to be:
	\ben 
	\item 
	\emph{Intrinsically tame} if the universal minimal $G$-space $M(G)$ is tame. 
	Equivalently, if  
	every continuous action of $G$ on a compact space $X$ 
	has a closed $G$-subspace $Y$ which is tame. 
	\item 
	\emph{Intrinsically c-ordered} 
	if $M(G)$ is a c-ordered $G$-system. 
	
	\item  
	\emph{Convexly intrinsically tame} 
	if the universal irreducible affine $G$-system $I\!A(G)$ is tame. 
	Equivalently, if every continuous affine action on a compact convex space $X$ has a closed 
	(not necessarily affine) 
	$G$-subspace $Y$ which is tame. 
	
	\item \emph{Convexly intrinsically c-ordered} if every continuous affine action on a compact convex space $X$ has a closed $G$-subspace $Y$ which is c-ordered.
	\item For brevity we use the following short names: \textit{int-tame, int-c-ord,  conv-int-tame, conv-int-c-ord.}
	
	\item If in (1) we consider only \emph{metrizable} spaces $X$ then we say that $G$ is \textit{metrically intrinsically tame} (in short: \textit{metr-int-tame}). 
	Similarly can be defined also 
	the notions of {\em{metrically intrinsically c-ordered}},  and {\em{metrically convexly intrinsically ordered}} groups.
	
	\een
\end{defin}

\begin{remark} \label{r:diagrams} 
	By results of \cite{GM-c} every c-ordered $G$-system is tame.  
	Thus, we have the following diagram of implications:
	\begin{equation*}
	\xymatrix
	{
		\text{extr. amenability}\  \ar@2{->} [d] \ar@2{->}[r]\  & \text{int-c-ord}\   \ar@2{->} [d] \ar@2{->}[r] &  \ \text{int-tame} \ar@2{->}[d] \ar@2{->}[r] &  \text{metr-int-tame}  \ar@2{->}[d]\  & \\
		\text{amenability}  \ar@2{->}[r] &  \text{conv-int-c-ord} \ar@2{->}[r] &   \text{conv-int-tame} \ar@2{->}[r] &  \text{metr-conv-int-tame}
	}
	\end{equation*}
\end{remark}

\begin{examples} \label{r:oldEx} \cite{GM-tLN} 
	\ben 
	\item $\mathrm{SL}_n(\R)$, $n >1$ 
	(more generally, any semisimple Lie group $G$ with finite center and no compact factors) 
	are conv-int-tame nonamenable topological groups which are not int-tame. Moreover,  $\mathrm{SL}_2(\R)$ is int-c-ordered.  
	\ssk 
	Sketch: 
	by Furstenberg's result \cite{Furst-63-Poisson} 
	the universal minimal strongly proximal system
	$M_{sp}(G)$ is the homogeneous space $X=G/P$, where $P$ is a minimal parabolic 
	subgroup (see \cite{Gl-book1}).
	Results of Ellis  and Akin 
	(Example \cite[Example 6.2.1]{GM-tLN})  
	show that the enveloping semigroup $E(G,X)$ in this case is a Rosenthal compact space, 
	whence the  system $(G,X)$ is tame by the dynamical BFT dichotomy (Theorem \ref{D-BFT}). 
	
	Note that in the case of $G=\mathrm{SL}_2(\R)$ it follows that
	in any compact \textit{affine} $G$-space one can find either a
	1-dimensional real projective $G$-space
	(a copy of the circle) or a fixed point.  	
	For general $\mathrm{SL}_n(\R)$ -- flag manifolds and their $G$-quotients. 	
	
	\item 
	The	Polish group $H_+(\T)$ is an int-c-ord nonamenable topological group. 
	Note that for $G=H_+(\T)$ every compact $G$-space $X$ contains 
	either a copy of $\T$, as a  $G$-subspace, or a $G$-fixed point. 
	\ssk 
	Sketch: 
	This follows from Pestov's theorem \cite{Pes98}, which
	identifies $(G,M(G))$ for $G = H_+(\T)$ as the tautological action of $G$ on $\T$.  
	
	\item
	The	Polish groups $\Aut(\mathbf{S}(2))$ and 
	$\Aut(\mathbf{S}(3))$ of automorphisms 
	of the circular directed graphs $\mathbf{S}(2)$ and $\mathbf{S}(3)$, are intrinsically c-ordered (hence, also int-tame). 
	The universal minimal $G$-systems for  $\Aut(\mathbf{S}(2))$ and  
	$\Aut(\mathbf{S}(3))$ are computed by L. Nguyen van Th\'{e} in \cite{van-the}. 
	
	\item 
	The Polish group $H(C)$, of homeomorphisms of the Cantor set,
	is not conv-int-tame.
	\item
	The Polish group $G=S_{\infty}$, of permutations of the natural numbers, 
	is amenable,  
	hence conv-int-tame but not int-tame. 
	\een
\end{examples}

\sk 

Given a class $\textbf{P}$ of compact $G$-systems one can define the notions  
``intrinsically $\textbf{P}$ group" and  ``convexly intrinsically 
$\textbf{P}$ group" in a manner analogous to the one we adopted for $\textbf{P}=$Tame. 

Recall that the following inclusion relations are valid 

\centerline{AP $\subset$ WAP $\subset$ HNS $\subset$ Tame} 
\nt where, AP = almost periodic (equivalently, equicontinuous) $G$-systems, 
WAP = weakly almost periodic systems, HNS = hereditarily nonsensitive.  
For the definitions of $Asp(G)$, HNS and WAP see for example \cite{GM-tLN}.
By the dynamical BFT dichotomy \cite{GM-rose, GM-tLN}, the class Tame of dynamical systems, plays a special role being, in a sense, the largest class of all ``small" systems. 

Note also that If $\textbf{P}$ is the class of all metrizable $G$-systems then 
$G$ is int-$\textbf{P}$ if and only if $M(G)$ is metrizable.

\begin{remark} \label{r:collapsing} 
	
	It turns out that 
	in this terminology a topological group is convexly intrinsically HNS 
	(and, hence, also conv-int WAP) iff it is amenable. 
	This follows from the fact that every HNS minimal $G$-system is almost periodic; 
	see \cite[Prop. 7.18]{Me-nz} and \cite[Lemma 9.2.3]{GM1}. 
	Thus we have
	
	\centerline{int-AP = int-WAP = int-HNS. } 
	
	\nt Also, by the left amenability of the algebra $Asp(G)$, 	
	which corresponds to the class of HNS systems,
	\cite{GM-fpt}, we get
	
	\centerline{amenability = conv-int-AP = conv-int-WAP = conv-int-HNS}
	
	\nt	
	
	\sk

	This ``collapsing effect" inside HNS and the exceptional role of tameness in 
	the dynamical BFT dichotomy 
	suggest that the  
	notion of convex intrinsic tameness is a natural generalization of amenability.
	This is also supported by several natural examples 
	(see Examples \ref{r:oldEx} and Corollary \ref{c:examples}). 
\end{remark}

\sk

\section{Circular order, topology and inverse limits} 
\label{s:c}

In this section we give some technical results about circular order which we use in Section \ref{Sec:main}. 
For more information and properties we refer to \cite{GM-c}. 

\begin{defin} \label{newC} \cite{Kok,Cech} 
	Let $X$ be a set. A ternary relation $R \subset X^3$ on $X$ is said to be a {\it circular} (or, sometimes, \emph{cyclic}) order  
	if the following four conditions are satisfied. It is convenient sometimes to write shortly $[a,b,c]$ instead of $(a,b,c) \in R$. 
	\ben
	\item Cyclicity: 
	$[a,b,c] \Rightarrow [b,c,a]$;  
	
	\item Asymmetry: 
	$[a,b,c] \Rightarrow (b, a, c) \notin R$; 
	
	\item Transitivity:    
	$
	\begin{cases}
	[a,b,c] \\
	[a,c,d]
	\end{cases}
	$ 
	$\Rightarrow [a,b,d]$;
	
	\item Totality: 
	if $a, b, c \in X$ are distinct, then \ $[a, b, c]$ 
	or $[a, c, b]$. 
	\een
\end{defin}

Observe that under this definition $[a,b,c]$ implies that $a,b,c$ are distinct. 

For 
$a,b \in X$ define the (oriented) \emph{intervals}: 
$$
(a,b)_R:=\{x \in X: [a,x,b]\}, \ \  [a,b]_R:=(a,b) \cup \{a,b\}, \ \ [a,b)_R:=(a,b) \cup \{a\}. 
$$
Sometimes we drop the subscript, or write $(a,b)_o$ when context is clear. 
Clearly, 
$X \setminus [a,b]_R=(b,a)_R$ for $a \neq b$ and $X \setminus [a,a]_R= X \setminus \{a\}$.


\begin{remark} \label{r:chech} \ \cite[page 35]{Cech}
	\ben 
	\item 
	Every linear order $<$ on $X$ defines a \emph{standard circular order} $R_{<}$ on $X$ as follows: 
	$[x,y,z]$ iff one of the following conditions is satisfied:
	$$x < y < z, \ y < z < x, \  z < x < y.$$	
	\item (cuts) Let $(X,R)$ be a c-ordered set and $z \in X$. 
	For every $z \in X$ the relation 
	$$z <_z a, \ \ \ \  a <_z b \Leftrightarrow [z,a,b] \ \ \ \forall a \neq b \neq z \neq a$$
	is a linear order on $X$ and $z$ is the least element. This linear order 
	restores the original circular order. Meaning that $R_{<_z}=R$. 
	\een
\end{remark}

The following two technical results are easy to verify.

\begin{prop} \label{Hausdorff}  \
	\ben 
	\item 
	For every c-order $R$ on $X$ the family of subsets
	$${\mathcal B}_1:=\{X \setminus [a,b]_R : \ a,b \in X\} \cup \{X\}$$  
	forms a base for a topology $\tau_R$ on $X$ which we call the \emph{interval topology} of $R$. 
	\item If $X$ contains at least three elements then the (smaller) family of intervals  
	$${\mathcal B}_2:=\{(a,b)_R : \  a,b \in X, a \neq b\}$$
	forms a base for the same topology $\tau_R$ on $X$. 
	\item The interval topology $\tau_R$ of every circular order $R$ is Hausdorff. 
	\een 
\end{prop} 

\sk
\begin{lem} \label{c-isclosed} 
	Let $R$ be a circular order on $X$ and $\tau_R$ the induced (Hausdorff) topology. 
	Then for every  $[a,b,c]$ there exist neighborhoods $U_1,U_2,U_3$ of $a,b,c$ respectively such that $[a',b',c']$ for every $(a',b',c') \in U_1 \times U_2 \times U_3.$  
\end{lem}
%
%
%
%
%
%
%
%
%
%
%
%
%
%
%


\sk 

Denote by $C_n:=\{1,2, \cdots, n\}$ the standard $n$-cycle with the natural circular order.  
Let $(X,R)$ be a c-ordered set. 
We say that a vector
$(x_1,x_2, \cdots, x_n) \in X^n$ is a \textit{cycle} in $X$ if it satisfies the following two conditions: 

\ben 
\item For every $[i,j,k]$ in $C_n$ and \textit{distinct} 
$x_i, x_j, x_k$ we have $[x_i,x_j,x_k]$; 
\item $x_i=x_k \ \Rightarrow$ \ 
$(x_i=x_{i+1}= \cdots =x_{k-1}=x_k) \ \vee \ (x_k=x_{k+1}=\cdots =x_{i-1}=x_i).$  
\een  
\textit{Injective cycle} means that all $x_i$ are distinct. 

\begin{defin} \label{c-ordMaps} Let $(X_1,R_1)$ and  $(X_2,R_2)$ be c-ordered sets. 
	A function $f: X_1 \to X_2$ is said to be {\it c-order preserving} 
	if $f$ moves every cycle to a cycle. 
	Equivalently, if it satisfies the following two conditions:
	
	\ben 
	\item For every $[a,b,c]$ in $X$ and \textit{distinct} 
	$f(a), f(b), f(c)$ we have $[f(a), f(b), f(c)]$; 
	\item If $f(a)=f(c)$ then $f$ is constant on one of the closed intervals $[a,c], [c,a]$. 
	\een 
	%
\end{defin}


\begin{lem} \label{invLimCirc}
	Let $X_{\infty}:=\underleftarrow{\lim} (X_i,I)$ be the inverse limit of the inverse system 
	$$\{f_{ij} \colon X_j \to X_i, \ \ i \leq j, \ i,j \in I\}$$ where $(I,\leq)$ 
	is a directed poset. Suppose that every $X_i$ is a c-ordered set with the c-order $R_i \subset X_i^3$ and each bonding map $f_{ij}$ is
	c-order preserving. On the inverse limit $X_{\infty}$ define a ternary relation $R$ as follows. 
	An ordered triple $(a,b,c) \in X_{\infty}^3$ belongs to $R$ iff $[p_i(a),p_i(b),p_i(c)]$ is in $R_i$ for some $i \in I$.  Then 
	$R$ is a c-order on $X_{\infty}$ and each projection map 
	$$p_i \colon X_{\infty} \to X_i, \ p_i(a)=a_i$$ is c-order preserving.
	
\end{lem}
\begin{proof} 
	We start with 
	\sk 
	
	\nt	\textbf{Claim 1}: \textit{For every three distinct elements $a,b,c \in X_{\infty}$ there exists a separating projection 
		$p_i \colon X_{\infty} \to X_i$; 
		that is,  
		$a_i = p_i(a) ,b_i =p_i(b), c_i=p_i(c)$ are distinct.} 
	\sk 
	
	Indeed, since $a,b,c \in X_{\infty}$ are distinct there exist indexes $j(a,b), j(a,c), j(b,c) \in I$ such that 
	$$a_{j(a,b)} \neq b_{j(a,b)}, \  a_{j(a,c)} \neq c_{j(a,c)}, \ b_{j(b,c)} \neq c_{j(b,c)}.$$
	
	Since $I$ is directed we may choose $i \in I$ which dominates all three indexes 
	$j(a,b)$, $j(a,c), j(b,c)$. Then $a_i,b_i,c_i$ are distinct. 
	
	\sk 
	
	\nt	\textbf{Claim 2}: \textit{If $[a_i,b_i,c_i]$ for some $i \in I$ and $a_j,b_j,c_j$ are distinct in $X_j$ for some $j \in J$ then $[a_j,b_j,c_j]$. }

	\sk 
	Indeed, choose an index $k \in I$ such that $i \leq k, j \leq k$ then $a_k,b_k,c_k$ are distinct. 
	Necessarily $[a_k,b_k,c_k]$. Otherwise, $[b_k,a_k,c_k]$ by the Totality axiom.  Then also $[b_i,a_i,c_i]$ because the bonding map $f_{ik}: X_k \to X_i$ is c-order preserving	and $a_i,b_i,c_i$ are distinct in $X_i$ (since $[a_i,b_i,c_i]$). 
	
	Since $[a_k,b_k,c_k]$ it follows that $[a_j,b_j,c_j]$ because the bonding map $f_{jk}: X_k \to X_j$ is c-order preserving. 
	
	\sk 

	Now we show that $R$ is a c-order (Definition \ref{newC})  on  $X_{\infty}$. 
	
	The Cyclicity axiom is trivial. 
	
	Asymmetry axiom is easy by Claim 2. 
	
	Transitivity: by Claims 1 and 2 there exists $k \in I$ such that $[a_k,b_k,c_k]$ and $[a_k,c_k,d_k]$. Hence, $[a_k,b_k,d_k]$ by the transitivity of $R_k$. Therefore, $[a,b,d]$ in $X_{\infty}$ by the definition of $R$. 
	
	Totality: 
	if $a, b, c \in X$ are distinct, then $a_j,b_j,c_j$ are distinct for some $j \in I$ by Claim 1. By the totality of $R_j$ we have $[a_j, b_j, c_j]$ $\vee$ $[a_j, c_j, b_j]$, 
	hence also $[a, b, c]$ $\vee$ $[a, c, b]$ in $R$.
	
	So, we proved that $R$ is a c-order on $X_{\infty}$. 
	
	\sk 
	Now we show that each projection $p_i \colon X_{\infty} \to X_i$ is c-order preserving.  Condition (1) of Definition \ref{c-ordMaps} is satisfied for every $i \in I$ 
	by Claim 2 and the definition of $R$. In order to verify condition (2) of Definition \ref{c-ordMaps} assume that $p_i(a)=p_i(b)$ for some distinct $a, b \in X_{\infty}$. 
	We have to show that $p_i$ is constant on one of the closed intervals $[a,b], [b,a]$. 
	If not then there exist $u,v \in X_{\infty}$ such that $[a,u,b], [b,v,a]$ but $p_i(u) \neq p_i(a) \ne p_i(v)$.  As in the proof of Claim 1 one may choose an index $k \in I$ such that the elements $p_k(a), p_k(b), p_k(u), p_k(v)$ are distinct in $X_k$. Then we get that the bonding map $f_{ik} \colon X_k \to X_i$ does not satisfy condition (2) of Definition \ref{c-ordMaps}.  
	This contradiction completes the proof.  
\end{proof}

\begin{lem} \label{invLimCircTop}
	In terms of Lemma \ref{invLimCirc} assume in addition that every $X_i$ is a compact c-ordered 
	space and each bonding map $f_{ij}$ is continuous. 
	Then the topological inverse limit $X_{\infty}$ is also a c-ordered (nonempty) 
	compact space.
\end{lem} 
\begin{proof} Let $\tau_{\infty}$ be the usual topology of the inverse limit $X_{\infty}$. 
	It is well known that the inverse limit $\tau_{\infty}$ of compact Hausdorff spaces is nonempty and compact Hausdorff. 
	Let $\tau_c$ be the interval topology  
	(see Proposition \ref{Hausdorff}) 
	of the c-order $R$ on $X_{\infty}$, 
	where $(X_{\infty},R)$ is 
	defined as in Lemma \ref{invLimCirc}. We have to show that $\tau_{\infty}=\tau_c$. Since $\tau_c$ is Hausdorff 
	it is enough to show that $\tau_{\infty} \supseteq \tau_c$. 
	This is equivalent to showing 
	that every interval $(u,v)_o$ is $\tau_{\infty}$-open in $X_{\infty}$ for every distinct $u,v \in X_{\infty}$, where   
	$$
	(u,v)_o:= \{x \in X| \ [u,x,v]\}. 
	$$ 
	Let $w \in (u,v)_o$; that is, $[u,w,v]$. By our definition of the c-order $R$ of $X_{\infty}$ we have $[u_i,w_i,v_i]$ in $X_i$ for some $i \in I$. 
	The interval $O_i:=(u_i,v_i)_o$  is open
	in $X_i$. Then its preimage $p_i^{-1}(O_i)$ is $\tau_{\infty}$-open in $X_{\infty}$. On the other hand, 
	$$w \in p_i^{-1}(O_i) \subseteq (u,v)_o.$$
	Indeed, if $x \in p_i^{-1}(O_i)$ then $p_i(x) \in (u_i,v_i)_o$. This means that 
	$[u_i,x_i,v_i]$ in $X_i$. By the definition of $R$ we get that $[u,x,v]$ in $X_{\infty}$. So, $x \in (u,v)_o$. 
\end{proof}

\begin{lem} \label{denseHom} \
	\ben 
	\item Let $\widehat{G}$ be the two-sided completion of a topological group $G$. 
	Then $G$ is int-tame (respectively, int-c-ordered) iff $\widehat{G}$ is int-tame (respectively, int-c-ordered).  
	\item 
	Let $h \colon G_1 \to G_2$ be a continuous dense homomorphism. 
	Then if $G_1$ is  int-tame (respectively, int-c-ordered), so is $G_2$.    
	\een
\end{lem}
\begin{proof} (1) 
	Let $G \curvearrowright X$ be a continuous action on a compact Hausdorff space $X$. 
	Then we have the continuous extended action 
	$\widehat{G} \curvearrowright X$ induced by the continuous homomorphism 
	$\g \colon G \to H(X)$, where $H(X)$ 
	(the full homeomorphism group) 
	is always two-sided complete. 
	Also, $M(G)=M(\widehat{G})$. 
	Now observe that $\widehat{G} \curvearrowright X$ is tame (c-ordered) iff 
	$G \curvearrowright X$ is tame (c-ordered). 
	Indeed, in the case of tameness, observe that 
	$cl_p(f \widehat{G}) = cl_p(f G)$
	for every $f \in C(X)$ 
	and use the characterization (see \cite[Proposition 5.6]{GM-rose}) of tame functions. 
	
	For the case of c-order, let $X$ be a c-ordered $G$-system. 
	It is enough to show that $X$ is a c-ordered $\widehat{G}$-system. 
	Let $t \in \widehat{G}$ and assume $[x,y,z]$. 
	We have to show that $[tx,ty,tz]$. 
	Assuming the contrary, 
	by the Totality axiom (Definition \ref{newC}) 
	we have $[ty,tx,tz]$. Choose a net $g_i$ in $G$ which tends to $t$. 
	Clearly, $[g_ix,g_iy,g_iz]$. Finally apply Lemma \ref{c-isclosed} and the Totality axiom. 
	
	(2)  
	Follows easily from (1) because $h(G_1)$ and $G_2$ have the ``same" two-sided  completion.  	
\end{proof}

\sk 
\section{Ultrahomogeneous actions on circularly ordered sets} 
\label{Sec:main}

Now we introduce the following definition, a natural circular version of the ultrahomogeneity for linear orders.

\begin{defin}
	We say that an 
	effective  
	action of a group $G$ of c-order automorphisms on a c-ordered infinite set $X$ is 
	\emph{ultrahomogeneous} if every c-order isomorphism between two finite subsets can be extended to a c-order automorphism of $X$. 
	Let us say that a circularly ordered set $X$ is ultrahomogeneous if the action $H_+(X) \curvearrowright X$ is ultrahomogeneous. 
\end{defin}

\begin{lem} \label{lem;dense}
	If a group $G$ acts ultrahomogeneously on 
	an infinite circularly ordered set $X$, 
	then the order type of $X$ is dense; i.e. there are no vacuous intervals $(a,b)_R$.
\end{lem}

\begin{proof}
	Let $z \in X$ be an arbitrary point and consider the linearly ordered set $(X, <_z)$
	(the ``cut" at $z$, Remark \ref{r:chech}.2).  
	Then $X \setminus \{z\}$ is a ultrahomogeneous linearly ordered set, whence is of the dense type.
	As $z$ is arbitrary our assertion follows. 
\end{proof}


%
%
%

Let $H$ be a closed subgroup of a topological group $G$. Denote by 
$$q \colon G \to G/H, \  g \mapsto gH=[g]$$ 
the natural (open) projection on the coset $G$-space $G/H$ endowed with the quotient topology. 
Recall that the topological space $G/H$ admits a natural \emph{right uniformity}  $\mu_r(G/H)$. A uniform basis of $\mu_r(G/H)$ is a family of all entourages of the 
form 
$$
\tilde{V}:=\{(xH,yH): \ xy^{-1} \in V\}
$$
where $V \in N_e(G)$ is a neighrborhood of the identity $e$ in $G$. 
Then $q$ is uniformly continuous and the Samuel compactification of $\mu_r(G/H)$ induces the maximal $G$-compactification  
$G/H \hookrightarrow \beta_G G/H$  of the $G$-space $G/H$ (which, in this case, is a topological embedding). 

Recall also that every uniform structure can be defined by uniform coverings.
In the case of $\mu_r(G/H)$ the corresponding basis is the following family of  uniform coverings of $G/H$ 
$$
\nu_V:=\{V[x]: \ [x]=xH \in G/H\} 
$$
where $V \in N_e(G)$. 


\sk 
The following result is a circular analog of Pestov's result.  
As we already mentioned the
``intrinsically linearly ordered" groups 
(that is, the groups with linearly ordered 
$M(G)$) 
are exactly 
the extremely amenable groups.
About the structure and properties of $M(G)$ see Theorem \ref{thm;ep}.

\begin{thm} \label{main-ultra} 
	Let a group $G$ act ultrahomogeneously 
	on an infinite circularly ordered (discrete) set $X$. 
	Then $G$, with the pointwise topology, 
	is intrinsically c-ordered (i.e., $M(G)$ is a c-ordered $G$-system). If $X$ is countable then $M(G)$ is metrizable. 
\end{thm}
\begin{proof}

	Choose any point $z \in X$. The corresponding stabilizer $H:=St(z)$ 
	is a clopen subgroup of $G$. 
	
	As in \cite[Proposition 2.3]{GM-c} consider the canonical linear order on $X$ defined by the rule: 
	$$z <_z a, \ \ \ \  a <_z b \Leftrightarrow [z,a,b] \ \ \ \forall a \neq b \neq z \neq a.$$
	By construction (the smallest element) $z$ is $H$-fixed.

	\sk 
	\nt \textbf{Claim.} The induced action $H \curvearrowright X \setminus \{z\}$ on the linearly ordered set $X \setminus \{z\}$ is linear order preserving and  ultrahomogeneous.
	\begin{proof} 
		Indeed, let $x <_z y$ in $X \setminus \{z\}$. Then $[z,x,y]$. For every $h \in H$ we have $[h(z),h(x),h(y)]=[z,h(x),h(y)]$. This means that $h(x) <_z h(y)$. 
		
		Furthermore, the action $H \curvearrowright X \setminus \{z\}$ is
		ultrahomogeneous.
		First of all the action $St(z) \curvearrowright X \setminus \{z\}$ is also effective. 
		For every pair of $k$-element chains 
		$$
		x_1 <_z x_2 <_z \cdots <_z x_k \ \ \ \ \ y_1 <_z y_2 <_z \cdots <_z y_k
		$$
		in $X_{<_z} \setminus \{z\}$ consider the pair of $(k+1)$-chains by adding the least element $z$. That is, consider 
		$$
		z <_z x_1 <_z x_2 <_z \cdots <_z x_k \ \ \ \ \ z<_z y_1 <_z y_2 <_z \cdots <_z y_k.
		$$
		We can treat them as a pair of cycles in the circularly ordered set $X_{<_z}$. The bijection 
		$x_i \mapsto y_i, \  z \mapsto z$ is an isomorphism between finite circularly ordered sets. Therefore there exists an extension $g: X \to X$, $g \in G$. Since $g(z)=z$ we have $g \in H$ and $g$ preserves the order $<_z$ on $X$ and hence also on $X \setminus \{z\}$. 
	\end{proof}

	\sk 
	By Pestov's theorem \cite{Pes98} mentioned above, $H$, 
	equipped with the pointwise topology on the discrete space $X \setminus \{z\}$ 
	is extremely amenable. 
	Note that this topology on $H=St(z)$ is the same as the topology of simple convergence on $X$ with respect to the action 
	of $H$ on the discrete set $X$ (since $z$ is fixed under $H$). 
	So, we may treat $H$ as 
	a topological subgroup of $G$.

	We have the natural isomorphism of $G$-actions 
	$$i \colon G/H \to X, \ gH \mapsto gz.$$ 
	For convenience sometimes we identify below the 
	discrete $G$-sets $G/H$ and $X$. 
	We can treat $i$ also as an isomorphism of uniform spaces, where $G/H$ us endowed with its right uniformity $\mu_r(G/H)$ and $X$ carries the natural uniformity $\mu_X$ such that the coverings of the form $\{V x: x \in X\}$, with $V$ a neighborhood of the identity in $G$, is a base of the uniformity. 

	Let $\beta_G (G/H)$ be the maximal $G$-compactification of the 
	discrete coset $G$-space $G/H:=\{gH:   g \in G\}$.  
	Since $H$ is extremely amenable one may apply 
	another result of 
	Pestov \cite[Theorem 6.2.9]{Pe-b} 
	which asserts that any minimal compact $G$-subsystem of $\beta_G (G/H)$ 
	is isomorphic to the universal system $M(G)$. Therefore, it is enough to show that $\beta_G (G/H)$  is 
	c-ordered (it is easy to show that a closed subspace of a c-ordered  compact space is again c-ordered).  
	Below we will see that $G/H$, as a uniform space with respect to the usual right uniformity, is precompact (Lemma  \ref{Claim1}). 
	So, in this case $\beta_G (G/H)$ is just the completion $\widehat{G/H}$ (and the Samuel compactification) of the precompact uniform $G$-space $G/H$. 
	
	So our aim is to show that the compact space $\widehat{G/H}$ is a c-ordered $G$-system. We show this in Lemma \ref{l:Limit} below. For these purposes we will need some preparations. 
	
	Let $F:=\{t_1,t_2, \cdots, t_m\}$ be an \textit{$m$-cycle} on $X$. That is, a c-order preserving injective map 
	$F: C_m \to X$, where $t_i=F(i)$ and $C_m:=\{1,2, \cdots, m\}$ with the natural circular order.  We have a natural equivalence "modulo-$m$" between $m$-cycles (with the same support). 
	
	For every given cycle $F:=\{t_1,t_2, \cdots, t_m\}$  
	define the corresponding finite disjoint covering $cov_F$ of $X$, by adding to the list: all 
	points $t_i$ and (nonempty by Lemma \ref{lem;dense}) intervals $(t_i,t_{i+1})_o$ between the cycle points. More precisely we consider the following disjoint cover which can be think of an equivalence relation on $X$. 
	$$
	cov_F:=\{t_1, (t_1,t_2)_o, t_2, (t_2,t_3)_o, \cdots, t_m, (t_m,t_1)_o \}. 
	$$ 
	Moreover, $cov_F$ naturally defines also a finite c-ordered set $X_F$ by ``gluing the points" of the interval $(t_i,t_{i+1})_o$ for each $i$. 
	So, the c-ordered set $X_F$ is the factor-set of the equivalence relation $cov_F$ and it contains $2m$ elements. 
	In the extremal case of $m=1$ (that is, for $F=\{t_1\}$) we define $cov_F:=\{t_1, X \setminus\{t_1\}\}$. 

	We have 
	the following canonical c-order preserving onto map  
	\begin{equation} \label{projection} 
	\pi_F \colon X \to X_F, \  \ 
	\pi_F(x) =
	\begin{cases}
	t_i & {\text{for}} \ x=t_i\\
	(t_i,t_{i+1})_o & {\text{for}} \ x \in (t_i,t_{i+1})_o,
	\end{cases}
	\end{equation}

	\begin{lem} \label{Claim1}
		The family 
		$
		\{cov_F\}
		$ 
		where $F$ runs over all finite injective cycles $$F: \{1,2,\cdots,m\} \to X=G/H$$ on $X$ is a bases of the natural uniformity of the coset right uniform space $G/H$.  
		This uniform structure is precompact.
	\end{lem}
	
	\begin{proof} First note the following general fact. If 
		$G$ is a nonarchimedean (NA) topological group and $H$ is its arbitrary subgroup. Then the uniform space $G/H$ is NA. 
		That is, uniformly continuous equivalence relations on $G/H$ 
		form a uniform bases of the right uniformity. 
		Indeed, for every open subgroup $P$ of $G$ we have an equivalence relation on $G$ -- partition by double cosets of the form $PxH$, where $x \in S(P,H)$ and $S(P,H) \subset G$ is a subset of representatives. This partition induces an equivalence relation on $G/H$ which is an element of the right uniformity that majorizes the basic uniform covering $\nu_P=\{P[x]: \ [x] \in G/H\}$.

		Back to our $G$ from Theorem \ref{main-ultra} which acts on the discrete set $X$. By the definition of pointwise topology $\tau_p$, basic neighborhoods of the identity $e \in G$ are 
		subgroups
		$$V_F=St(F) : = \cap_{x \in F} St(x),$$ where $F$ is a finite subset of $X$.
		Accordingly, we have a basic uniform covering $\{V_Fx: \ x \in X\}$ of $X$. We can suppose that $F=\{t_1,t_2, \cdots, t_m\}$ is a cycle of $X$. Since the action is ultrahomogeneous it follows that $V_F x=(t_i,t_{i+1})_o$ for every $x \in (t_i,t_{i+1})_o$ (and $V_F t_i=t_i$). Hence we obtain that 
		$cov_F=\{V_Fx: \ x \in X\}$ is, in fact, a finite covering. 
		It follows immediately from the definitions that $G/H=X$, as a uniform structure, is precompact meaning that its completion $\widehat{G/H}$ is compact.
	\end{proof} 
	
	%
	%
	%
	%

	Let $Cycl(X)$ be the set of all finite injective cycles.  
	Every finite $m$-element subset $A \subset X$ defines a cycle $F_A \colon \{1,\cdots,m\} \to X$  (perhaps after some reordering) which is uniquely defined up to the natural cyclic equivalence and the image of $F_A$ is $A$. 
	
	$Cycl(X)$ 
	is a poset if we define $F_1 \leq F_2$ whenever $F_1 \colon C_{m_1} \to X$ is a \emph{sub-cycle} of $F_2 \colon C_{m_2} \to X$. This means that $m_1 \leq m_2$ and $F_1(C_{m_1}) \subseteq F_2(C_{m_2})$.   
	This partial order is directed. Indeed, for $F_1,F_2$ we can consider $F_3=F_1 \bigsqcup F_2$ whose support is the union of the supports of $F_1$ and $F_2$. 
	
	For every $F \in Cycl(X)$ we have the disjoint finite 
	$\mu_X$-uniform covering $cov_F=\{V_Fx:  x \in X\}$ of $X$. 
	As before we can look at $cov_F$ as a c-ordered (finite) set 
	$X_F$. Also, as in equation \ref{projection}, we have a c-order preserving natural map $\pi_F: X \to X_F$ which are uniformly continuous into the finite (discrete) uniform space $X_F$. 
	Moreover, 
	if $F_1 \leq F_2$ then $cov_{F_2} \subseteq cov_{F_1}$. This implies that the equivalence relation $cov_{F_2}$ is sharper than $cov_{F_1}$.  
	We have a c-order preserving (continuous) onto bonding map $f_{F_1,F_2} \colon X_{F_2} \to X_{F_1}$ between finite c-ordered sets. Moreover, $f_{F_1,F_2} \circ \pi_{F_2}=\pi_{F_1}$. 

	In this way we get an inverse system 
	$$
	\{f_{F_1,F_2} \colon X_{F_2} \to X_{F_1}, \ \ F_1 \leq F_2, \ \ I \}
	$$
	where $(I,\leq)=Cycl(X)$ be the directed poset defined above. 
	It is easy to see that $f_{F,F}=id$ and $f_{F_1,F_3}=f_{F_1,F_2} \circ f_{F_2,F_3}$  for every $F_1 \leq F_2 \leq F_3$.

	Denote by $$X_{\infty}:=\underleftarrow{\lim} (X_F, \ I)  \subset \prod_{F \in I} X_F$$ the corresponding inverse limit. 
	Its typical element is 
	$\{(x_F) : F \in Cycl(X)\} \in X_{\infty}$, where $x_F \in X_F$.  The set 
	$X_{\infty}$ carries a circular order $R$ as in Lemma \ref{invLimCirc}. 
	
	For every given $g \in G$ (it is c-order preserving on $X$) we have the induced isomorphism $X_F \to X_{gF}$ of c-ordered sets, where $t_i \mapsto gt_i$ and $(t_i,t_{i+1})_o \mapsto (gt_i,gt_{i+1})_o$ for every $t_i \in cov_F$. 
	For every 
	$F_1 \leq F_2$ we have $f_{F_1,F_2} (x_{F_2})=x_{F_1}$. This implies that $f_{gF_1,gF_2} (x_{gF_2})=x_{gF_1}$. So, $(gx_F)=(x_{gF}) \in X_{\infty}$. 
	
	Therefore $g \colon X \to X$ can be extended canonically to a 
	map 
	$$
	g_{\infty} \colon X_{\infty} \to X_{\infty}, \ \ g_{\infty} (x_F) := (x_{gF}).
	$$ 
	\nt This map is a c-order automorphism. Indeed, if $[x,y,z]$ in $X_{\infty}$ then there exists $F \in I$ such that $[x_F,y_F,z_F]$ in $X_F$. Since $g \colon X \to X$ is a c-order automorphism we obtain that $[gx_{F},gy_{F},gz_{F}]$ in $X_{gF}$. 
	
	One may easily show that we have a continuous action $G \times X_{\infty} \to X_{\infty}$, where $X_{\infty}$ carries the compact topology of the inverse limit as a closed subset of the topological product $\prod_{F \in I} X_F$ of finite discrete spaces $X_F$.

	\begin{lem} \label{l:Limit} 
		$\beta_G (G/H)$=	$\widehat{G/H} \simeq X_{\infty}$ as compact $G$-spaces. 
		Furthermore, if $X$ is countable then $\widehat{G/H}$ is a metrizable compact. 
	\end{lem}
	\begin{proof} Recall that the map 
		$i \colon G/H \to X, \ gH \mapsto gz$ identifies 
		the discrete $G$-spaces $G/H$ and $X$. 
		Moreover, as our discussion above shows, $i$ is a uniform isomorphism of 
		the precompact uniform spaces $(G/H,\mu_r)$ and $(X,\mu_X)$, where $\mu_X$ can be treated as the weak uniformity with respect to the family of maps $\{\pi_F \colon X \to X_F: \ F \in Cycl(X)\}$ (into the finite uniform spaces $X_F$). 
		
		Observe that $f_{F_1,F_2} \circ \pi_{F_2}=\pi_{F_1}$ for every $F_1 \leq F_2$.  By the universal property of the inverse limit we have the canonical uniformly continuous map $\pi_{\infty} \colon X \to X_{\infty}$. It is easy to see that it is an embedding of uniform spaces and 
		that $\pi_{\infty}(X)$ is dense in $X_{\infty}$. Since $X=G/H$ is a precompact uniform space we obtain that its uniform completion is a compact space and can be identified with $X_{\infty}$. The uniform embedding $G/H \to X_{\infty}$ is a $G$-map. It follows that the uniform isomorphism $\widehat{G/H} \to X_{\infty}$ is also a $G$-map.  
	\end{proof}
	
	\sk 
	On the other hand this inverse limit $X_{\infty}$ is c-ordered as it follows from Lemmas \ref{invLimCirc} and \ref{invLimCircTop}.  Furthermore, as we have
	already mentioned the action of $G$ on $X_{\infty}$ is c-order preserving. 
	Therefore $X_\infty = \widehat{G/H}$ is a compact c-ordered $G$-system. 
	Since $M(G)$ is its subsystem this completes the proof of 
	the first part of
	Theorem \ref{main-ultra}.  
	Note that  $G$ is always nonarchimedean, being 
	a topological subgroup of the symmetric group $S_X$. 
	When $X$ is countable $S_X \cong S_\N$ is a Polish group and 
	the group $(G,\tau_p)$ is separable and metrizable. 
	Moreover, analyzing the proof, we see that in that case $\widehat{G/H}$ and  
	$M(G)$ are metrizable.  
\end{proof}
 
In fact, $M(G)$ coincides with $\beta_G (G/H)$, as it follows from Theorem \ref{thm;ep} below. 

\sk \sk 

Let $\T:=\R / \Z$ be the circle and $q \colon \R \to \R / \Z$ is the canonical projection. 
Denote by $\Q_{\circ}$ the set $q(\Q)$ of all rationals  into the circle $\T$ (endowed with the natural c-order). Let $D_{\circ}=D/\Z$ be its subset $q(D)$, where $D$ is a set of all dyadic rationals.  
Denote by $H_+(\T)$ the group of all circular order preserving homeomorphisms of the circle $\T$.

\begin{cor} \label{c:examples} 
	The following topological groups are intrinsically circularly ordered. Hence, also (convexly) intrinsically tame. 
	\ben 
	\item Polish group $G:=\Aut(\Q_{\circ})$. 
	Furthermore, $M(G)$ is metrizable.  
	
	\item \emph{Thompson's circular group} $T$ with the pointwise convergence topology (acting on $D_{\circ}$); 
	
	\item Topological group (not Polish) $H_+(\T_{discr})$ in the pointwise topology $\tau_p$  with respect to the action $H_+(\T) \curvearrowright , (\T,\tau_{discr})$, where $\T$ carries the discrete topology; 
	\item Polish group $H_+(\T)$ in the usual compact open topology. 
	\een
	
\end{cor}
\begin{proof}
	
	(1) $\Q_{\circ}$ is ultrahomogeneous as a c-ordered set. 
	
	(2) The action of the Thompson's group $T$ on $D_{\circ}$ is ultrahomogeneous. 
	
	
	(3) The action $H_+(\T) \curvearrowright \T$ is ultrahomogeneous.
	
	(4) 
	$(H_+(\T, \tau_{discr}), \tau_p)$ is intrinsically c-ordered by (3). Now observe that we have a well defined continuous onto injective homomorphism 
	$H_+(\T, \tau_{discr}) \to H_+(\T)$ 
	and apply Lemma \ref{denseHom}.2. 
\end{proof}

\begin{remark}
	It is a standard fact that Thompson group $F$ acts ultrahomogeneously on 
	the linearly ordered set 
	$D$,  \cite[Lemma 4.2]{CFP}. 
	So, by Pestov's theorem $F$, as a topological group in the pointwise topology, is extremely amenable. 
	Recall that it is an open problem whether the discrete group $F$ is amenable. 
	
	It seems also to be well known (or, to follow easily from \cite[Lemma 4.2]{CFP}, as was observed by G. Golan) 
	that Thompson's circular group $T$ acts ultrahomogeneously on $D / \Z$. See also \cite[Remark 6]{GGJ}.  
\end{remark}


\sk 
We need to discuss more properties of the $G$-system $X_\infty = \beta_G (G/H)$, the inverse limit system constructed in Theorem \ref{main-ultra}. 
	By our construction $X$ is a discrete space densely embedded into the compact 
	space $X_\infty$. 
	First we observe that  $u \in X_\infty$ is an isolated point of $X_\infty$ if and only if $u \in X$.
	Indeed, it is well known that every dense locally compact subspace of a compact space is open. Therefore, $X$ is an open subset of $X_\infty$. It follows that every $x \in X$ is an isolated point of $X_\infty$. 
	On the other hand, no other point $y$ of $X_\infty$ is not isolated. Indeed, if yes, then the open subset $\{y\}$ intersects the dense subset $X$.

	\sk
	Let $Z = X_\infty \setminus X$. 
	Then $Z$ is a closed (hence, compact) $G$-invariant subset of $X_\infty$.   
	Recall that $$X_{\infty}:=\underleftarrow{\lim} (X_F, \ I)  \subset \prod_{F \in I} X_F$$ is an inverse limit. 
	Its elements $u \in X_F$
	are ``threads"
	$u=\{(u_F) : F \in Cycl(X)\}$, where $u_F \in X_F$, 
	$F:=\{t_1,t_2, \cdots, t_m\} \subset (X,\circ)$ is a cycle, 
	$X_F$ is a 
	finite disjoint (cyclically ordered) cover of $X$
	$$
	X_F:=\{t_1, (t_1,t_2)_o, t_2, (t_2,t_3)_o, \cdots, t_m, (t_m,t_1)_o \}  
	$$ 
	and $q_F \colon X_{\infty} \to X_F$ is the natural c-order preserving projection. 
	
	A thread $u \in X_{\infty}$ represents an element $x \in X$ iff there exists $F  \in Cycl(X)$ such that $u_F=t_i=x$ 
	for some $t_i \in F$ (it is the only case when $u$ is isolated).  
	
	A typical (basic) neighborhood of $u \in X_{\infty}$ is 
	$$
	O_F(u):=\{v \in X_{\infty}: v_F = u_F\}.  
	$$
	Here $u_F$ is one of the members in $X_F$. If $u \in Z$ (i.e., $u \notin X$) then necessarily $u_F=(t_i,t_{i+1})_\circ$ for some $i$ and then $u \in (t_i,t_{i+1})_{\infty}$ (that is, $[t_i,u,t_{i+1}]$ in $X_{\infty}$). 
	In fact, taking into account that $q_F \colon X_{\infty} \to X_F$ is c-order preserving, we have 
	\begin{equation} \label{eq:nbd} 
		O_F(u)=(t_i,t_{i+1})_{\infty}. 
	\end{equation}
	

\sk 
For every $x \in X$ define $x^- \in X_{\infty}$ as a thread $u=(u_F) \in X_{\infty}$ such that for every cycle $F$ which contains $t_i=x$, we have $u_F=(t_{i-1},t_i)$. 
This defines one and only one element from $X_{\infty}$ having this property. The map $X \to X^-, x \mapsto x^-$ is a bijection.  
Dually (replacing $t_i$ by $t_{i-1}$) can be defined $x^+$. Set $X^-=\{x^-:x \in X\}$, 
$X^+=\{x^+:x \in X\}$ and $X^*=X^- \cup X \cup X^+$. It is easy to see that all three subsets are pairwise disjoint. 

\sk 
\begin{defin}	\cite{Gl-book1} 
	A compact $G$-system $Z$ is called: 
	\begin{enumerate}  
		\item
		{\em Strongly proximal} 
		if for every  $\mu \in \mathcal{P}(Z)$, the compact space of probability measures on $Z$ equipped 
		with its weak-star topology, there are a net $g_\al$  in $G$ and a point $z \in Z$ such that $\lim g_\al \mu = \del_z$,
		the point mass at $z$.
		\item
		{\em Extremely proximal} 
		if $Z$ is infinite and for every nonempty closed subset $A \subset Z$ with $A \not= Z$, there is a net
		$g_\al$ in $G$ and a point $z \in Z$, such that, in the space $2^Z$ of closed subsets of $Z$ equipped with the
		Vietoris topology, $\lim g_\al A = \{z\}$.
	\end{enumerate}
\end{defin}

It is shown in \cite[Proposition VII, 3.5]{Gl-book1} that if a group $G$ admits a nontrivial minimal extremely proximal action
on a compact space $Z$ then $G$ contains a free subgroup on two generators.

\sk 
\begin{thm} \label{thm;ep}
	Let a group $G$ act ultrahomogeneously on an infinite circularly ordered set $X$.  
	Then the $G$-system  $Z=X_\infty \setminus X$ is: 
	\begin{enumerate}
		\item
		Minimal, 
		 hence $M(G, \tau_p)=X_\infty \setminus X$. 
		\item
		Extremely proximal. 
		\item
		The universal strongly proximal minimal system $M_{sp}(G)$ of $G$. 
		\sk 
		Moreover, 
		\item 
		$G$ contains a free subgroup on two generators.
		\item 
		The universal irreducible affine $G$-system $I\!A(G)$ is $P(X_\infty \setminus X)$ (the affine $G$-system of all probability measures on $X_\infty \setminus X$).  
		
	\end{enumerate}
\end{thm}

\begin{proof} 
		(1)
	Let $u, v \in Z$. Consider a basic neighborhood $O_F(u)$ in $X_{\infty}$. We will show that there exists $g \in G$ such that $gv \in O_F(u)$ (then necessarily $gv \in O_F(u) \cap Z$). By our construction, $u \in (t_i,t_{i+1})_{\infty}$. Take any pair $a, b \in X$ such that $v \in (a,b)_{\infty}$. By the ultrahomogeneity there exists $g \in G$ such that $ga=t_i, gb=t_{i+1}$. Then (since the $g$-translation  preserves the circular order) 
	$[a,v,b]$ implies $[t_i,gv,t_{i+1}]$. This means that $gv \in (t_i,t_{i+1})_{\infty}=O_F(u)$, as desired.  

	(2)  
	Let $A$ be a closed subset of $Z$ such that $A \neq Z$. Then there exists a pair
	$a, b \in X$ such that $A \subset [a,b]_{\infty}$. 
	Indeed, choose $u \in Z \setminus A$. Then $Z \setminus A$ is an open neighborhood of $u$ in $Z$. There exists a basic neighbourhood $O_F(u)$ in $X_{\infty}$ 
	(where $O_F(u):=\{v \in X_{\infty}: v_F = u_F\}$), 
	such that $O_F(u) \cap Z \subset Z \setminus A$. Since $u \notin X$, we have $u_F=(b,a)_{\infty}$ 
	for some $a,b \in X$. 
	So, $(b,a)_{\infty} \subset Z \setminus A$. Since $(b,a)_{\infty} \cup [a,b]_{\infty} = X_{\infty}$, we obtain  $A \subset [a,b]_{\infty}$. 
	
	\sk  
	\nt	\textbf{Claim}: 
	\textit{Let $F \in Cycl(X,\circ)$ and 
		$O_F(u)=(x_1,x_2)_{\infty}$ be a basic neighborhood of 
		$u=x_2^- \in X_{\infty}$ (where $x_1,x_2 \in X$). 
		Then there exists $a_F \in X$ such that 
		$$
		a_F \in (x_1,x_2^-)_{\infty}, \ \ x_2^- \in (a_F,x_2)_{\infty}.
		$$ }
	\begin{proof}  
		By definitions of $x_1^+, x_2^-$ and $[x_1,x_2^-,x_2]$,  there exists a sufficiently large 
		$F=\{t_1,t_2, \cdots, t_m\}  \in Cycl(X)$ 
		such that for some subcycle  $[i,k,j]$ we have  
		$$t_i=x_1,  \ \ u_F=(t_{k}, t_j)_{\circ}, \ \ t_j=x_2.$$ 
		Now, define $a_F:=t_{k}$. 
	\end{proof}
	Choose any $x_0 \in X$. Then $x_0^- \in Z$. 
	By Claim there exists a net 
	$$\{a_F: F \in Cycl(X)\} \ \ \text{where} \ a_F \in X$$ 
	such that $\cap_F (a_F, x_0)_{\infty} = \{x_0^-\}$ and $\lim_F a_F=x_0^-$. 
	By the ultrahomogeneity, 
	there is a net  $\{g_F \in G\}$ such that $g_F(a)=a_F, g_F(b)=x_0$.  Then it follows that $\lim_F g_F (A) = \{x_0^-\}$ in the Vietoris topology.

	(3) 
	Let $\mu$ be a probability measure on $X_\infty \setminus X$. 
	Let $\mu = \mu_a + \mu_c$
	be the decomposition of $\mu$ as a sum of an atomic and continuous measures.
	By proximality (which certainly follows from extreme proximality)
	there is a net in $G$ which shrinks $\mu_a$ to a single point mass (with the appropriate
	weight).
	So we now assume that $\mu$ is continuous. 
	By regularity, given an $\ep >0$
	there is a compact set $A \subset X_\infty$ with $1 -\ep < \mu(A) < 1$.
	Applying extreme proximality to $A$ we can find a net $g_\al \in G$
	with $\lim_\al A = \{z\}$. It is not hard to see that this implies
	that $\nu =\lim g_\al \mu$ (which by compactness we can assume exists)
	has the property that $\nu(V) \ge 1- \ep$  for any open neighborhood $V$ of $z$. 
	As $\ep$ is arbitrary this concludes the proof of strong proximality.
	
	Since the universal system $M(G)=X_\infty \setminus X$ is strongly proximal, then it is the 
	\emph{universal strongly proximal minimal system} of $G$. 
	
	(4) By (2) we may apply \cite[Prop. VII, 3.5]{Gl-book1}.
	
	(5)  Apply results of \cite{Gl-book1} making use of (3). 
\end{proof}

\sk 
\begin{cor}\label{cor:1-1}
	In the context of Theorem \ref{thm;ep} suppose that $X$ is countable.
	Then any nontrivial factor map $\pi \colon M(G) \to Y$ is an almost one to one extension;
	i.e. for a dense $G_\del$ subset $Z_0 \subset M(G)$ we have
	$\pi^{-1}(\pi(z)) = \{z\}$ for every $z \in Z_0$.
\end{cor}

\begin{proof}
	In that case $M(G)$ is metric and if $\pi : M(G) \to Y$ 
	is a nontrivial factor map and $R_\pi = \{(z, z'): \  \pi(z) = \pi(z')\}$,
	then the set-valued map $\pi^{-1} : Y \to 2^{M(G)}$,
	being upper-semi-continuous, has a dense $G_\del$ subset
	$Y_0 \subset Y$ of continuity points. Let $Z_0 = \pi^{-1}(Y_0)$,
	a dense $G_\del$ subset of $M(G)$.
	Suppose $z_0 \in Z_0$ and let $z_1$ be an arbitrary point of $M(G)$.
	Then $R_\pi[z_1] $ is a proper closed subset of $M(G)$ and, by minimality and extreme proximality
	of $M(G)$, there is a sequence
	$g_i \in G$ with $\lim g_i z_1 = z_0$ and $\lim g_i R_\pi[z_1] = \{z_0\}$.
	However, as $z_0$ is a continuity point, we also have $\lim g_i R_\pi[z_1] = R_\pi[z_0]$,
	whence $R_\pi[z_0] = \{z_0\}$.
\end{proof}

\sk  
\section{The Fra\"{i}ss\'{e} class of finite circularly ordered systems and the KPT theory} 
\label{Fraisse}

In this section we present an alternative proof of Theorem \ref{main-ultra}, in the countable case.
It employs the notion of Fra\"{i}ss\'{e} classes and the Kechris-Pestov-Todorcevic theory \cite{KPT} and
has the advantage that it automatically yields an explicit description of $M(G)$.

In the sequel we will freely use the notations and results from this seminal work.
We begin by citing the following theorem proved in \cite[Theorem 7.5]{KPT}.

\begin{thm}\label{Ram}
	Let $L \supseteq \{<\}$ be a signature, $L_0 = L \setminus \{<\}$, 
	$\mathcal{K}$ a reasonable Fra\"{i}ss\'{e} order class in $L$. 
	Let $\mathcal{K}_0 = \mathcal{K} | L_0$, and $\mathbf{F} = {\rm Flim}(\mathcal{K})$, 
	$\mathbf{F}_0 = {\rm Flim}(\mathcal{K}_0)$. 
	Let $G = \Aut(\mathbf{F})$, $G_0 = \Aut(\mathbf{F}_0)$.
	Let $X_{\mathcal{K}}$ be the set of linear orderings 
	$\prec$ on $F (= F_0)$ which are $\mathcal{K}$-admissible.
	\begin{enumerate}
		\item
		If $\mathcal{K}$ has the Ramsey property, the $G_0$-ambit 
		$(X_{\mathcal{K}}, \prec_0)$ is the universal $G_0$-ambit with the property that 
		$G$ stabilizes the distinguished point, i.e. it can be mapped homomorphically to any 
		$G_0$-ambit $(X,x_0)$ with $G \cdot x_0 = \{x_0\}$. 
		Thus any minimal subflow of $X_{\mathcal{K}}$ is the universal minimal flow of $G_0$. 
		In particular, the universal minimal flow of $G_0$ is metrizable.
		\item
		If $\mathcal{K}$ has the Ramsey and ordering properties, $X_{\mathcal{K}}$ 
		is the universal minimal flow of $G_0$.
	\end{enumerate}
\end{thm}

If $(A, [\cdot, \cdot, \cdot])$ is a finite circularlly ordered set
we say that a linear order on $A$ is {\em compatible} when it is
obtained from $(A, [\cdot, \cdot, \cdot])$ by a cut (see \cite[Remark 2.4]{GM-c}.). 
In the context of the following theorem this is the same as being admissible in the sense of \cite[Definition 7.2]{KPT}.
Thus $X_{\mathcal{K}}$ can be identified with the orbit closure of 
$\prec_0 $, $\overline{G_0 \cdot \prec_0} \subset LO$,
the compact set of linear orders on $F$, the universe of $\mathbf{F}$.

For the definition and properties of $\Sp(X; \Q_\circ)$ see 
\cite[Lemma 2.11]{GM-c}.

\begin{thm}
	Let $L = \{ [\cdot, \cdot, \cdot], <\}$ be a signature and $L_0 = L \setminus \{<\}$.
	Let  $\mathcal{K}$ be the Fra\"{i}ss\'{e} order class in $L$ consisting of all the finite
	circularlly ordered sets with compatible linear order. 
	Let $\mathbf{F} = {\rm Flim}(\mathcal{K})$, $\mathbf{F}_0 = {\rm Flim}(\mathcal{K}_0)$ and 
	set $G = \Aut(\mathbf{F})$, $G_0 = \Aut(\mathbf{F}_0)$.
	Then 
	\begin{enumerate}
		\item
		$\mathcal{K}$ is a reasonable order class and it satisfies the Ramsey and ordering properties. 
		\item
		Let $\Q_\circ$ denote the set $\Q/\Z \subset \R/\Z=\T$ of roots of unity,
		equipped with the cyclic order relation $[\cdot, \cdot, \cdot]$ it inherits from $\R/\Z$.
		Then $\mathbf{F} =  {\rm Flim}(\mathcal{K}) = (\Q_\circ, [\cdot, \cdot, \cdot], \prec_0)$,
		where $\prec_0$ is the dense linear order on  $\Q_0$ induced by the linear order
		obtained from the $\al$-cut $\Q_0 \cap (\al, \al +1) \pmod 1$ for some (any) $\al \in \R \setminus \Q_0$.
		\item
		
		$M(G_0) = X_{\mathcal{K}} \cong  
		\Sp(\T; \Q_\circ)$ is the universal minimal $G_0$-flow.
	\end{enumerate}
\end{thm} 

\begin{proof}
	(1)\  
	It is easy to see that $\mathcal{K}$ is a reasonable order class. It is also easy to check that it has the ordering property.
	The fact that $\mathcal{K}$ has the Ramsey property follows by the classical Ramsey theorem. 
	
	(2) Clearly $(\Q_\circ, [\cdot, \cdot, \cdot], \prec_0)$ is ultrahomogeneous and 
	$\mathcal{K} = {\rm Age}(\mathbf{F})$.
	Now apply  Fra\"{i}ss\'{e}'s theorem \cite[Theorem 2.2]{KPT} to conclude that
	$\mathbf{F} = (\Q_\circ, [\cdot, \cdot, \cdot], \prec_0)$.
	
	(3)\
	By Theorem \ref{Ram} we have 
	$$
	M(G_0) = X_\mathcal{K} =  \overline{G_0 \cdot \prec_0} \subset LO(\Q_0).
	$$
	Define a map $\Phi \colon \Sp(X; \Q_\circ) \to X_\mathcal{K}$ as follows.
	For $\xi \in \T \setminus \Q_0$ let $\Phi(\xi) = \prec_\xi$ be the linear order defined by the
	$\xi$-cut $\Q_0 \cap (\xi, \xi +1) \pmod 1$. For $\eta \in \Q_0$ let $\Phi(\eta^-) = \prec_\eta^-$ be the linear order 
	defined by $\Q_0 \cap [\eta, \eta +1) \pmod 1$, i.e. the dense order with $\eta$ as a first element;
	and let $\Phi(\eta^+) = \prec_\eta^+$ be the linear order 
	defined by $\Q_0 \cap (\eta, \eta +1] \pmod 1$, i.e. the dense order with $\eta$ as the last element.
	It is easy to check that $\Phi$ defines a $G_0$-flow isomorphism.
\end{proof}


\begin{remark} \label{R:Factors} 
	Here $G_0$ can be identified with $\Aut(\Q_{\circ})$. 
	With this explicit presentation of $M(G_0)$ it is easy to see 
	that the canonical map
	$\Sp(\T; \Q_\circ) \to \T$ is the only nontrivial factor of the $G_0$-flow $M(G_0)$. 
	It follows that every continuous action of the Polish group $(\Aut(\Q_{\circ}),\tau_p)$ on a compact space $X$ 
	has a closed $G$-subspace $Y$ which is c-ordered. 
\end{remark}

\begin{remark}
	It is not hard to check that, in the case that $X$ is countable, the collection
	$Cycl(X)$ of finite injective cycles on $X$ forms a 
	countable {\em projective Fra\"{i}ss\'{e} family of finite topological L-structures}
	in the sense of Irvin and Solecki \cite{IS} (here $L$ is the relational language $L = \{R\}$).
	Of course then $X_\infty = \widehat{G/H}$ is the corresponding 
	projective Fra\"{i}ss\'{e} limit of $Cycl(X)$.
\end{remark}

\sk 

\section{Automatic continuity and Roelcke precompactness}\label{AC}

\subsection{Automatic continuity and metr-int-c-ord}


\begin{lem} \label{lem-automat} 
	The Polish group $\Aut(\Q_{\circ})$ has the automatic continuity property. That is, every group homomorphism $h \colon \Aut(\Q_{\circ}) \to G$ to a separable 
	topological group $G$ is continuous. 
\end{lem}
\begin{proof} It is a well known theorem by Rosendal and Solecki \cite{RS} that the Polish group  $\Aut(\Q_{<})$ has the automatic continuity property. The same is true for the larger group $\Aut(\Q_{\circ})$ because the Polish group 
	$\Aut(\Q_{<})$ is an open subgroup of  $\Aut(\Q_{\circ})$. 
\end{proof}

\begin{cor} \label{metr-conv-int} 
	For every action of $G:=\Aut(\Q_{\circ})$ by homeomorphisms on a \textit{metric} compact space there exists a closed c-ordered $G$ subsystem. 
	That is, the discrete group $\Aut(\Q_{\circ})$ is metrically intrinsically c-ordered.  
\end{cor}
\begin{proof}
	This follows from Theorem \ref{main-ultra} and Lemma \ref{lem-automat} 
	taking into account Remark \ref{R:Factors}. 
\end{proof}

\sk 

\subsection{Roelcke precompactness and Kazhdan's property (T)} 

\begin{prop} \label{t:Rolelckeprecompact}
	The Polish group $G:=\Aut(\Q_{\circ})$ is Roelcke precompact. 
\end{prop}
\begin{proof}
	It is welll known that the Polish subgroup $H:=\Aut(\Q_{<})$ is Roelcke precompact. By our construction in Theorem \ref{main-ultra} the coset uniform space $(G/H,\mu_r)$ is 
	precompact. Now apply Proposition  \ref{l:Roelcke}. See also Proposition \ref{pr:RP} for another proof in a more general case. 
\end{proof}

The following result generalizes \cite[Prop. 9.17]{RD}.  

\begin{prop} \label{l:Roelcke}
	Let $H$ be a subgroup of a topological group $G$ such that $H$ is Roelcke precompact as a topological group and the coset space $G/H$ is precompact in the standard right uniformity. Then $G$ is Roelcke precompact. 
\end{prop}
\begin{proof}
	We have to show that for every 
	neighborhood  $U$ of $e$ in $G$ there exists a finite subset $F \subset G$ such that $G=UFU$. In fact, it is enough to show that $G=U^2FU$.
	Since $G/H$ is precompact with respect to the standard right uniformity there exist finitely many points $a_1H, a_2H, \cdots, a_nH$ in the coset space $G/H=\{[g]=gH\}_{g \in G}$ such that $E:=\{a_1, a_2, \cdots, a_n\}$ is a subset of $G$ and 
	$$
	G/H= \cup \{U [a_i]: a_i \in E\}. 
	$$ 
	This implies that 
	$$
	G= U E H. 
	$$
	Now, for $G=U^2FU$ it suffices to show that $EH \subset UFU$ for some finite $F \subset G$. 
	
	By our assumption $H$ is Roelcke precompact as a topological group. That is, $(H,R_H)$ is precompact. 
	It follows that $H$ is a precompact subset of $G$ with respect to the Roelcke uniformity $R_G$ of $G$. Then the same is true for each subset $a_iH \subset G$. Hence the finite union $EH$ is also Roelcke precompact subset of $G$. Therefore, for given $U$ there exists a finite subset $F$ in $G$ such that 
	$$EH \subset UFU, $$ 	as required. 
\end{proof}

\sk 
\begin{cor} \label{c:kazhdan} 
	$\Aut(\Q_{\circ})$ satisfies the Kazhdan's property (T). 
\end{cor}
\begin{proof}
	By a result of Evans and Tsankov \cite{ET} every nonarchimedean Roelcke precompact Polish group 
	has the Kazhdan's property (T). 
\end{proof}


\begin{prop} \label{pr:RP} 
	Let $G \curvearrowright X$ be a ultrahomogeneous action on a circularly (linearly) ordered set $X$. Then the topological group $(G,\tau_p)$ is Roelcke precompact.
\end{prop}
\begin{proof}
	First consider the case of linearly ordered $X$. Then Roelcke precompactness of $G$ can be proved as in the proof of Rosendal \cite[Theorem 5.2]{RosendalBergman09} (for countable $X=\Q$). 
	See also Tsankov \cite{Tsankov}. 
	One may show that for every open subgroup $U$ of $G$ there are only
	finitely many different double cosets  
	$UxU$ in $G$. 
	
	In the case of circularly ordered $X$ we combine Lemma \ref{l:Roelcke}.
	and Theorem \ref{main-ultra} which asserts that the coset space $G/H$ is right precompact, where $H$ is a stabilizer subgroup $St(z)$ which is Roelcke precompact by the linear order case.
\end{proof}


\begin{remark} \ 
	\begin{enumerate}
		\item Propositions \ref{t:Rolelckeprecompact} and \ref{pr:RP} seemingly can be derived also by results of Tsankov \cite{Tsankov}. 
		\item The Polish group $G:=\Aut(\Q_{\circ})$ has strong uncountable cofinality as was proved by Rosendal \cite[Theorem 7.1]{RosendalBergman09}. Hence (again by the results of Rosendal \cite{12RosendalBounded}) has properties (OB) and (ACR). The latter means that any affine continuous action
		of $G$ on a separable reflexive Banach space has a fixed point. 
	\end{enumerate}  
\end{remark}

\sk 

\section{Some perspectives and questions}

\subsection{Conv-int-tame nonamenable discrete groups} 

For nondiscrete topological groups we have several examples of conv-int-tame nonamenable groups. Among others, the 
Lie 
groups $\mathrm{SL}_n(\R)$ and 
the Polish group
$H_+(\T)$ (see Example \ref{r:oldEx}). However, the discrete case is open. 

\begin{question} \label{ourQ} 
	Is there a nonamenable convexly intrinsically tame (countable) \textit{discrete} group ? 
\end{question}




Note that a (topological) group $G$ is conv-int-tame iff its universal minimal strongly proximal $G$-system system $M_{sp}(G)$ is tame. 
Thinking about possible candidates we observe that the free group $F_2$ is not conv-int-tame,
because $M_{sp}(F_2)$ is not tame. 
One way to see this is as follows.
Let $X$ be the Cantor set, $(X,T)$ be the Morse minimal system 
and $S \colon X \to X$ a self homeomorphism such that the cascade $(X,S)$ is strongly proximal.
Let $F_2$, the free group on generators $a, b$, act on $X$ via the map $F_2 \to \Homeo(X)$,
defined by $a \mapsto T, b \mapsto S$. Clearly then the $F_2$-system $X$
is minimal, strongly proximal but not tame, since already the subaction 
$(X,T)$ 
is not tame. 


%


\subsection{Topological subgroups of $S_{\infty}$}  

\begin{problem}
	Kechris-Pestov-Todorcevic result \cite[Theorem 4.8]{KPT}  characterizes extremely amenable subgroups of $S_{\infty}$ in terms of Fra\"{i}ss\'{e} limits. Namely, 
	$\Aut(A)$ is extremely amenable (where $A$ is the Fra\"{i}ss\'{e} limit of a class $K$) if and only if the Fra\"{i}ss\'{e} order class $K$ has the Ramsey property. 
	
	
	It would be interesting to study (or, even characterize) 
	some other classes of Polish groups (or of closed subgroups of $S_{\infty}$) $G$,
	for which  $M(G)$ is :
	\ben 
	\item int-c-ordered.
	\item int-tame.
	\item conv-int-c-ordered.
	\item conv-int-tame. 
	\item metric versions of the previous concepts (as in Definition \ref{d:int-tame}). 
	\een 
	
\end{problem}

The class (1) contains the Polish group $\Aut(\Q_{\circ})$ 
(which is not extremely amenable). The class (4) contains all countable discrete amenable groups.  

\sk 

\subsection{Other structures} 

Pestov's theorem can be reformulated by saying that if $G$ acts ultrahomogeneously on a linearly ordered set then for the topological group $(G,\tau_p)$ in its pointwise convergence topology 
the minimal universal $G$-system 
$M(G)$
is also linearly ordered. 
Indeed, note that every compact minimal linearly ordered $G$-space is trivial. 

As we have seen Pestov's theorem can be naturally generalized to ultrahomogeneous circularly ordered sets. Namely, $M(G)$ is again circularly ordered. It is natural to wonder if there are some analogs for other structures. 

\begin{question} \label{q:str} 
	Let $G$ act 
	on a set $X$ with some  structure $R$. Under which reasonable conditions the universal minimal $G$-system $M(G,\tau_p)$ also admits a structure of the same type ?  
\end{question}

\sk

\section{Appendix: Large ultrahomogeneous circularly ordered sets}

It is well known that 
for every infinite cardinal $\tau$ there exists a ultrahomogeneous 
linearly ordered set $X$ of cardinality $\tau$  (see for example \cite{Pes98}).  
As expected, the same is true for circularly ordered sets. 
For the countable case we have the unique (up to isomorphism) model $\Q_{\circ}$. 
For the cardinality $2^{\aleph_0}$ we have, at least, the circle $\T$. 
As to the general case, 
very recently, responding to our question, J.K. Truss \cite{Tr} and V.G. Pestov \cite{Pest-pr18}  informed us that,  according to their (unpublished) notes, the following result holds
(their approaches are essentially different).

\begin{prop} \label{p:largeCO} (Pestov \cite{Pest-pr18}, Truss \cite{Tr})
	For any prescribed infinite cardinal $\tau$ there exists a ultrahomogeneous circularly ordered set of cardinality $\tau$. 
\end{prop}

With his permission we reproduce Pestov's proof of Proposition \ref{p:largeCO}. 

\begin{proof} 
	
	
	{\bf Step 1:}
	
	For every infinite cardinal $\tau$, there exists a linearly ordered field having $\tau$ as its cardinality, 
	this is well known and was rediscovered a number of times \cite{HM, N}. 
	
	The simplest 
	construction is as follows (borrowed from \cite{LPT}).
	Given a field $k$, denote $k(\alpha)$ a simple transcendental extension of $k$ with variable $\alpha$. In other words, $k(\alpha)$ consists of all rational functions $p(\alpha)/q(\alpha)$ where $p,q$ are polynomials in $\alpha$ with coefficients in $k$. 
	If now $k$ is an ordered field, then $k(\alpha)$ becomes an ordered field with $\alpha$ as a positive infinitesimal, that is, $0<\alpha<x$ for all $x\in k$, $x>0$. The sign of a polynomial $p(\alpha)=a_0+a_1\alpha +\ldots +a_n\alpha^n$ is the sign of the non-zero coefficient of the lowest degree. This extends to the rational functions in an obvious way. Clearly, $k$ is an ordered subfield of $k(\alpha)$.
	
	If $\tau$ is an infinite cardinal, we construct recursively in $\beta\leq\tau$ an increasing transfinite sequence $k_{\beta}$ of ordered fields, where $k_0=\Q$ (or any other fixed ordered field), $k_{\beta+1}=k_\beta(\alpha_{\beta})$ with $\alpha_{\beta}$ being a positive infinitesimal over $k_{\beta}$, and for limit cardinals $\beta$ we set $k_{\beta}=\cup_{\gamma<\beta}k_{\gamma}$. It is easy to see that the ordered field $k_{\tau}$ has the required cardinality $\tau$. 
	
	\ssk
	
	{\bf Step 2:}
	
	Every ordered field $k$ has characteristic zero. The absolute value in $k$ is defined as $\abs x=\max\{x,-x\}$. 
	Given an ordered field $k$, an element $x\in k$ is {\em finite} if for some natural number $n$ one has $\abs x\leq n$. 
	The subset $\mathrm{fin}(k)$ of all finite elements of an ordered field forms a convex subring, in which $\Z$ is a cofinal and coinitial ordered subring.
	
	On the additive factor-group $\mathrm{fin}(k)/\Z$ one may define now a circular order as a ternary relation $R$. The argument is similar to the case of the circle $\R/\Z$. 
	A triple $(a,b,c)$ is in $R$ if and only if one can write $a=a^\prime+\Z$, $b=b^\prime+\Z$, $c=c^\prime+\Z$ with $a^\prime\leq b^\prime\leq c^\prime$. 
	
	Clearly, the cardinality of $\mathrm{fin}(k)/\Z$ equals that of $k$.
	
	\ssk
	
	
	{\bf Step 3:}
	
	To verify ultrahomogeneity of $\mathrm{fin}(k)/\Z$, let $a_1,\ldots,a_n$, and $b_1,\ldots,b_n$, be two $n$-tuples of elements of the group which are positively cyclically ordered, that is, whenever $[i,j,k]$ within the finite group $\Z_k$, one has $[a_i,a_j,a_k]$ and $[b_i,b_j,b_k]$. Rotating both sets (and thus preserving the cyclic order), one can assume without loss in generality that $a_1=b_1=0$. Identifying the group $\mathrm{fin}(k)/\Z$ with the interval $[0,1)$ in $k$, one gets two sets of representatives of the $n$-tuples within $k$, $0=a^\prime_1<a^\prime_2<\ldots <a^\prime_n<1$ and $0=b^\prime_1<b^\prime_2<\ldots<b^\prime_n<1$. 
	Now, like in the proof of Assertion 5.1 in \cite{Pes98}, apply a piecewise linear, order preserving bijective transformation of $[0,1)_k$ onto itself sending $a^\prime_i\mapsto b^\prime_i$, $i=1,2,\ldots, n$:
	\[f(x)=\begin{cases} \frac{b^\prime_2}{a^\prime_2}x, & \text{if $0\leq x \leq a^\prime_1$,} \\
	b^\prime_i+\frac{b^\prime_{i+1}-b^\prime_i}{a^\prime_{i+1}-a^\prime_i}(x-a^\prime_i), & \text{if $a^\prime_i\leq x\leq a^\prime_{i+1}$,
		$i=1,2,\dots,n-1$,} \\
	b^\prime_n+\frac{1-b^\prime_i}{1-a^\prime_i}(x-a^\prime_n), & \text{if $a^\prime_n\leq x <1$.}
	\end{cases}
	\]
	The resulting map lifts to a cyclical-order preserving self-bijection of the group, taking the first $n$-tuple to the second (strictly speaking, we have to compose this map with the two rotations).  
\end{proof}

\bibliographystyle{amsplain}

%
%
%
%

\end{document}